\documentclass{amsart}
\usepackage{mathtools}
\mathtoolsset{showonlyrefs}

\newtheorem{theorem}{Theorem}

\newtheorem{lemma}{Lemma}

\def \R {\mathbb R}
\def \C {\mathbb C}
\def \N {\mathbb N}
\def \S {\mathcal S}
\renewcommand {\chi}{\operatorname{1}}

\title{New criteria on positive-definite distributions}

\author{J. Haddad}

\begin{document}
\maketitle
\begin{abstract}
	We establish several sufficient conditions under which a locally-integrable function $f:\R^n \to \R$ represents a positive-definite distribution.
	In particular we consider functions of the form $f(\|x\|)$ where $\|\cdot\|$ is a fixed norm in $\R^n$.
\end{abstract}

\section{Introduction}

\subsection{Positive-definite functions}
A continuous function $f:\R^n \to \C$ is said to be positive-definite if it is the Fourier transform of a non-negative finite measure $\mu$ on $\R^n$, this is,
	\begin{equation}
		\label{def_fourier}
		\hat f(\xi) = \int_{\R^n} e^{-i (x, \xi) } d \mu(x).
	\end{equation}
	where $(\cdot, \cdot)$ denotes the inner product of $\R^n$.
Positive-definite functions and their generalization to the theory of distributions (see below) arise naturally in many areas of mathematical analysis.
For example, in probability theory, positive-definite functions are the characteristic functions of probability measures.
In integral equations they are the kernels of positive, translation invariant integral operators.
We refer to \cite{stewart1976positive} for a historical overview on positive-definite functions, its generalizations and its relevance to other areas of mathematics.

Equivalently, a continuous function $f$ is positive-definite if and only if for every $k \in \N, c_1, \ldots, c_k \in \C, x_1, \ldots, x_k \in \R^n$,
\begin{equation}
	\label{def_matrix}
	\sum_{i,j=1}^k c_i \overline{c_j} f(x_i - x_j) \geq 0,
\end{equation}
meaning that the matrix $\{f(x_i - x_j)\}_{i,j}$ is positive-definite.
This is usually taken as a definition of positive-definiteness, and the equivalence between the two properties is the content of Bochner's theorem (see \cite[Theorem 2, Section 3.2]{gelfandvilenkin}).

Great efforts have been made to determine whether a given function or distribution is positive-definite.
The Schoenberg problem asks for which values of $n,p,q$ is the function $f_{n,p,q}(x) = e^{-\|x\|_p^q}$ positive-definite, where $\|\cdot\|_p$ is the $\ell_p$ norm in $\R^n$.
This has important consequences in embedding problems of Banach spaces.
Schoenberg showed that a Banach space $(\R^n, \|\cdot\|)$ can be embedded in the Hilbert cube, if and only if $e^{-\|x\|^2}$ is positive-definite.
Moreover, given $0< q < 2$, the positive-definiteness of $e^{-\|x\|^q}$ is equivalent to the fact that the Banach space $(\R^n, \|\cdot\|)$ can be embedded as a subspace of $L_q([0,1])$ (see \cite[Theorem 6.6]{koldobsky2005fourier}).
Partial answers to the Schoenberg problem were obtained along the 20-th century, and the problem was finally settled by Koldobsky (see \cite{koldobsky1991schoenbergoriginal}, \cite{koldobsky1992schoenberg} and the references therein).

\subsection{Positive-definite distributions}
Let $\S$ denote the Schwarz space, consisting of $C^\infty$ functions $\varphi:\R^n \to \R$ such that for every $k \in \N$,
\[p_k(\varphi) = \sup_{|\alpha| \leq k} \sup_{x\in\R^n} \left| \frac{\partial^{|\alpha|} \varphi}{\partial x^\alpha} (x) (1+|x|)^k \right| < \infty,\]
where $|x| = \sqrt{(x, x)}$ is the Euclidean norm.
The dual $\S'$ is called the space of distributions on $\R^n$.
It consists of the linear functionals $f:\S \to \R$ which are continuous with respect to the family of seminorms $p_k$.
We denote by $\langle f, \varphi \rangle$ the action of the distribution $f \in \S'$ on the function $\varphi \in \S$.

A locally integrable function $g$ defined almost everywhere on $\R^n$ defines naturally a distribution by integration, as
\[\langle g, \varphi \rangle = \int_{\R^n} g(x) \varphi(x) dx.\]
Likewise, a locally finite measure $\mu$ also defines a distribution by the analogue formula
\[\langle \mu, \varphi \rangle = \int_{\R^n} \varphi(x) d \mu(x).\]

For $\varphi \in \S$, its Fourier transform, defined by
\[\hat \varphi(\xi) = \int_{\R^n} \varphi(x) e^{-i (x, \xi) } d x,\]
is also in $\S$. This fact allows us to define the Fourier transform of a distribution $f \in \S'$, as the distribution $\hat f$ given by
\[\langle \hat f, \varphi \rangle = \langle f, \hat \varphi \rangle.\]

A distribution $f \in \S'$ is said to be positive if $\langle f, \varphi \rangle \geq 0$ for every non-negative $\varphi \in \S$.
Bochner's theorem, generalized by Schwarz (see \cite[Theorem 2, Section 2.2]{gelfandvilenkin}) states that a distribution is positive-definite if and only if
it is the Fourier transform of the distribution defined by a non-negative tempered measure, this is, a measure $\mu$ for which 
\[\int_{\R^n} (1+|x|)^{-\beta} d\mu(x) < \infty,\]
for some number $\beta > 0$.

Of course the condition \eqref{def_matrix}, which is usually taken as the definition of positive-definiteness for functions, makes no sense for general distributions.
While positive-definite functions satisfy several convenient properties like being continuous and bounded, positive-definite distributions can be unbounded, even if they are defined by locally integrable functions.
For example, the function $|x|^{-1}$ is a positive-definite distribution but not a positive-definite function.

Positive-definite distributions played a fundamental role in the solution of the Busemann-Petty problem on sections of convex bodies \cite[Chapter 5]{koldobsky2005fourier}.
A (symmetric) convex body $K$ is the closed unit ball of a norm in $\R^n$ which we denote by $\|\cdot\|_K$.
The Busemann-Petty problem led to the study of intersection bodies, and later to its analytical characterization given by Koldobsky \cite{koldobsky2005fourier}.
The intersection body of a star body $L$ is a star body $K$ such that for every $v \in S^{n-1}$, $\|v\|_K^{-1}$ is the $(n-1)$-dimensional volume of the section of $L$ with the orthogonal space, $L \cap v^\perp$.
In general a convex body is called an {\it intersection body} if it is the limit in the Hausdorff metric of intersection bodies of convex bodies.
Koldobsky showed \cite{koldobsky2005fourier} that a symmetric convex body $K$ is an intersection body if and only if $\|\cdot\|_K^{-1}$ is a positive-definite distribution.

Eaton \cite{eaton1981projections} showed that for a random vector $X$, all linear combinations of its coordinates have the same distribution, up to a constant, if and only if its characteristic function is of the form $f(\|x\|_K)$.
Thus by Bochner's theorem, the problem of finding all random vectors with this property is equivalent to characterize the class of functions $f$ for which $f(\|x\|_K)$ is positive-definite. Koldobsky showed \cite{koldobsky2011positive} that for continuous $f:[0, \infty) \to \R$ under mild conditions, if $f(\|x\|_K)$ is positive-definite then the space $(\R^n, \|\cdot\|_K)$ embeds in $L_0$ (see \cite{koldobsky2011positive, kalton2007geometry} for details).

\subsection{Known criteria for positive-definiteness}
P\'olya's criterion states that an even function $f:\R \to \R$ which is convex on the positive real line and $\lim_{x \to \pm \infty} f(x) = 0$, is positive-definite (see \cite{polya1949remarks}).
This includes for example the function $e^{-|x|^p}$ for $p \in (0,1]$.

In fact, the function $x\in\R \mapsto e^{-|x|^p}$ is positive-definite for $p \in (0,2)$, since it can be written as the Laplace transform of a positive measure.
Moreover, if $f \in \S'$ is a distribution defined by
\[\langle f, \varphi \rangle = \int_I \langle f_\lambda, \varphi \rangle d\mu(\lambda),\]
where $f_\lambda$ is a family of positive-definite distributions, and $\mu$ is a finite measure in a set $I$, then $f$ is a positive-definite distribution.

Additionally, the sum, product and convolution of two positive-definite functions, are again positive-definite.

In \cite{koldobsky2002derivatives} Koldobsky showed the following:
\begin{theorem}[{\cite[Theorem 3]{koldobsky2002derivatives}}]
	\label{thm_subspace}
	Let $1 \leq k \leq n$ and let $f$ be an even function on $\R^n$, which is locally integrable, continuous and bounded outside of the origin. Suppose that the function $|x|^{-k} f(x)$ is locally integrable on $\R^n$. If $f|_H$ is a positive-definite distribution on $H$ for every $n-k$ dimensional subspace $H$ of $\R^n$, then $|x|^{-k} f(x)$ is a positive definite distribution on $\R^n$.
\end{theorem}
This shows, for example, that the function
\begin{equation}
	\label{eq_goodexample}
	g_{n,p}(x) = |x|^{1-n} e^{-|x|^p},
\end{equation}
is a positive-definite distribution for $p \in (0,2)$.

In this paper we give some new criteria to determine if a given real function on $\R^n$ is a positive-definite distribution.
We will show that in fact this is the case of $g_{n,p}$ for every $p>0$.

\subsection{Main Results}

Regarding functions that are radially decreasing, we obtain:
\begin{theorem}
	\label{thm_decreasing_dimn}
	Let $f:(0,\infty) \to \R$ be a non-negative non-increasing function.
	Assume any of the following two conditions are satisfied:
	\begin{itemize}
		\item $n \geq 3$ and $f(r) \min\{1, r\}$ is integrable.
		\item $n \geq 9$ and $r f(r)$ is locally integrable.
	\end{itemize}
	Then the function $|x|^{-n+2} f(|x|)$ is a positive-definite distribution.
\end{theorem}

Regarding functions of the form $f(\|x\|_K)$ we have:
\begin{theorem}
	\label{thm_omega}
	Let $f:(0,\infty) \to \R^+$ be an absolutely continuous non-increasing function such that $-t^n f'(t)$ is bounded and integrable, and such that $-t^{n-1} f'(t)$ is non-increasing.
	Then for every symmetric convex body $K$, the function $f(\|x\|_K)$ is a positive-definite distribution.
\end{theorem}

For $n=1, K=[-1,1]$, the conditions on $f$ are simply that $f$ should be convex (plus some decay condition), recovering in this way the P\'olya criterion.

Finally, we would like to highlight an interesting result regarding the positivity of certain inner products, which was obtained as a by-product of our methods.
\begin{theorem}
	\label{thm_convex_curiosity}
	Let $\varphi, \psi:\R^n \to \R$ be even integrable functions such that $\varphi$ as convex sub-level sets and $\psi$ is radially decreasing.
	Then
	\[\int_{\R^n} |x|^\alpha \varphi(x) \hat \psi(x) dx \geq 0\]
	for every $\alpha \in (-n, 2-n]$.
	In particular if $n=2$, 
	\[\int_{\R^2} \varphi(x) \hat \psi(x) dx \geq 0.\]
\end{theorem}

\subsection{Examples}

As mentioned before, Theorem \ref{thm_subspace} with $k=n-1$ shows that the function $g_{n,p}$ given in \eqref{eq_goodexample} is positive-definite for $p \in (0,2)$.
Theorem \ref{thm_decreasing_dimn} shows actually that $g_{n,p}$ is positive-definite in the full range $p > 0$, for $n \geq 3$.
For $n=2, p>0$ this is still true, but a completely different proof is needed and we will skip this case for simplicity.

The Fourier transform of $g_{n,p}$ is given in \cite[Example 4.7]{koldobsky2023comparison} as an example of a continuous function which is not an intersection function.
Theorem \ref{thm_decreasing_dimn} shows that $\hat g_{n,p}$ is also non-negative thus filling this small gap in the example.

Theorem \ref{thm_decreasing_dimn} also shows that Theorem \ref{thm_subspace} does not admit a converse, at least for $k = n-1, n\geq 3$: there exist functions $f:\R \to \R$ which are not positive-definite, such that $|x|^{-n+1} f(|x|)$ is positive-definite. Just take $f(t) = e^{-|t|^p}$ for any $p > 2$.

Now we give some examples related to Theorem \ref{thm_omega}, although they are admittedly artificial.
For any symmetric convex body $K$ and $a>0, \alpha \in (1-n, 2-n)$, the function
	\[
	F(x) = 
		\begin{cases}
			\|x\|_K^\alpha & \|x\|_K \leq a \\
			0 & \|x\|_K \geq a 
		\end{cases},
	\]
	is a positive-definite distribution.

For any symmetric convex body $K$, the function $F(x) = \|x\|_K^p$ is a positive-definite distribution for $p \in (-n, 1-n)$.
This is a very particular case of \cite[Theorem 3.18]{koldobsky2005fourier}, in which the conclusion holds for the range $p \in (-n, 3-n)$.

\section{Proofs of the main theorems}
The next lemma summarizes many situations where an inner product is non-negative.
\begin{lemma}
	\label{lem_decreasing_dim1}
	Let $\varphi, \psi:\R \setminus \{0\} \to \R$ be two non-negative, even, measurable functions, such that $\varphi(x)$ and $|x|^{-1} \psi(x)$ are non-increasing for $x>0$.
	Assume any of the three following conditions are met:
	\begin{itemize}
		\item $\varphi, \psi \in L^p(\R)$ with $p \in [1,2]$.
		\item $\varphi$ is bounded with compact support and $\psi(x) \min\{1, |x|^{-1}\}$ is integrable.
		\item $ \varphi \in L^1(\R) \cap C^3(\R)$ and $\psi \in L^1_{loc}(\R)$.
	\end{itemize}
	Then 
	\begin{equation}
		\label{eq_positive_product}
		\int_{-\infty}^{\infty} \hat \varphi(x) \psi(x) dx = \int_{-\infty}^{\infty} \varphi(x) \hat \psi(x) dx \geq 0.
	\end{equation}

\end{lemma}
Observe also that if $\psi$ is non-increasing, then clearly $|x|^{-1} \psi(x)$ is also non-increasing for $x>0$.
\begin{proof}
	Lets assume first that $\varphi(x) = \chi_{[-a,a]}(x)$ and $\psi(x) = |x| \chi_{[-b,b]}(x)$
	where $a, b \in [0,\infty)$.
	We have $\hat \varphi(\xi) = \frac {2 \sin(a \xi)}\xi$.
	\begin{align}
		\int_{-\infty}^{\infty} \hat \varphi(x) \psi(x) d x 
		&= 2 \int_{-\infty}^{\infty} \frac {\sin(a \xi)}\xi |\xi| \chi_{[-b,b]}(\xi) d\xi \\
		&= 4 \int_0^b \sin(a \xi) d\xi \\
		&= \frac 4a (1- \cos(a b)),
	\end{align}
	which is non-negative.

	Now assume $\psi(x) = |x| \delta(x)$ where $\delta$ is non-increasing, and $\delta, \varphi$ are even, bounded with compact support.
	Use the layer-cake formula to write
	\[\varphi(x) = \int_0^\infty \chi_{[-a(t),a(t)]}(x) dt,\quad \psi(x) = |x| \int_0^\infty \chi_{[-b(t), b(t)]}(x) dt,\]
	for some positive and non-increasing functions $a,b:[0,\infty) \to [0,\infty)$ which are also bounded and with compact support.

	By Fubini we get
	\[\int_{-\infty}^{\infty} \hat \varphi(x) \psi(x) dx = \int_0^\infty \int_0^\infty \int_{-\infty}^{\infty} 2 \frac{\sin(a(t)x)}{x} |x| \chi_{[-b(s), b(s)]}(x) dx dt ds\]
	which is non-negative.
	Notice that the conditions imposed to $\delta$ and $\varphi$ guarantee that we may change the order of integration.

	Now let us prove \eqref{eq_positive_product} under the first condition, namely, that $\varphi, \psi \in L^p$. Approximate $\varphi,\psi$ using
	\[\varphi_n(x) = \min\{n, \varphi(x)\} \chi_{[-n,n]}(x), \quad \psi_n(x) = \min\{n x, \psi(x)\} \chi_{[-n,n]}(x).\]
	Notice that $\psi_n(x) = \delta_n(x) |x|$ where $\delta_n$ is even, non-increasing, bounded and has compact support.
	By dominated convergence, $\varphi_n \to \varphi$, $\psi_n \to \psi$ in $L^p$.

	We have 
	\[\int_{-\infty}^\infty \hat \varphi_n(x) \psi_n(x) dx \geq 0.\]
	But since the left-hand side is a bounded bilinear form in $L^p \times L^p$, we conclude that
	\[\int_{-\infty}^\infty \hat \varphi(x) \psi(x) dx \geq 0,\]
	and the proof of the first part is complete.

	Now assume $\varphi, \psi$ satisfy the second condition, this is, $\varphi$ is bounded with compact support, and $\delta(x) = |x|^{-1} \psi(x)$ is integrable.
	The assumptions on $\varphi$ guarantee that $\hat\varphi(x) = \int_0^{\max \varphi} 2 \frac {\sin(a(t) x)}x dt$ from which we deduce that $|\hat\varphi(x)| \leq C \min\{1, |x|^{-1}\}$ for some constant $C$. By the assumption on $\psi$, the product $\hat\varphi \psi$ is integrable. We compute
	\begin{align}
		\int_{-\infty}^\infty \hat\varphi(x) \psi(x) d x 
		&= \int_{-\infty}^\infty \int_0^{\max \varphi} 2 \frac {\sin(a(t) x)}x |x| \delta(x) d t d x\\
		&= \int_0^\infty \int_0^{\max \varphi} 4 a(t) \frac {\sin(a(t) x)}{a(t) x} x \delta(x) d t d x,
	\end{align}
	and we see that the integrand is bounded in absolute value by $4 (\max a) \chi_{[0,\max \varphi]}(t) x \delta(x)$ for $x \in [0,1]$ and by $\chi_{[0,\max \varphi]}(t) \delta(x)$ for $x > 1$. This is, it is bounded by an integrable function in $\R \times \R$.
	This allows us to use the Fubini theorem, and conclude exactly as in the first part of the proof.

	For the last condition, approximate $\psi$ by $\psi_n = \psi \chi_{[-n,n]}$. Since $\varphi,\psi_n \in L^1$,
	\[\int_{-\infty}^\infty \hat \varphi(x) \psi(x) dx \geq 0.\]
	The hypothesis on $\varphi$ implies that $x \hat \varphi(x) \in L^1(\R)$, then 
	\[|\hat \varphi(x) \psi(x)| \leq \begin{cases} \|\varphi\|_1 \psi(x) & x \in [-1,1] \\ |\hat \varphi(x) x| \psi(1) & x \not\in [-1,1] \end{cases},\]
which is integrable. We conclude again by dominated convergence.

\end{proof}

The Radon transform of a Schwartz function $\varphi$ is the function $\mathcal R \varphi: S^{n-1} \times \R \to \R$ given by
\[ \mathcal R \varphi(v,t) = \int_{v^\perp + t v} \varphi.\]
Integration occurs in the hyperplane orthogonal to $v$ passing through the point $t v$, with respect to the $(n-1)$-dimensional Lebesgue measure.
It is easy to prove that the Fourier transform of $t \mapsto \mathcal R \varphi(v,t)$ is the function $t \mapsto \hat\varphi(t v)$.

If $K \subseteq \R$ is a symmetric convex body, the indicator function of $K$ is
\[1_K(x) = \begin{cases} 1 & \text{ if } x \in K \\ 0 & \text{ otherwise } \end{cases}.\]
	Its Radon transform is called the parallel section function
\[A_{K,v}(t) = \mathcal R 1_K(v,t).\]
By the Brunn's concavity principle, $A_{K,v}(t)$ is a decreasing function of $t$, for $t>0$.

\begin{proof}[Proof of Theorem \ref{thm_decreasing_dimn}]
	First let us consider the case where $f(r) \min\{1, r\}$ is an integrable function.
	Extend $f:\R \setminus\{0\} \to \R$ as an even function and consider an even non-negative test function $\delta$. By polar coordinates,
\begin{align}
	\int_{\R^n} |x|^{-n+2} f(|x|) \hat \delta(x) d x
	&= \int_{S^{n-1}} \int_0^\infty r f(r) \hat \delta(r v) d r d v \\
	&= \frac 12 \int_{-\infty}^\infty r f(r) \int_{S^{n-1}} \hat \delta(r v) d v d r \\
	&= \frac 12 \int_{-\infty}^\infty r f(r) \mathcal F\left[ \int_{S^{n-1}}\mathcal R \delta (v,\cdot) dv \right](r) d r.
\end{align}
The integral of the Radon transform can be written as
	\begin{equation}
		\label{eq_integral_radon}
		\int_{S^{n-1}}\mathcal R \delta (v,r) dv = \kappa_{n-2} \int_{\R^n} \delta(x) ( 1 - (r/|x|)^2 )_+^{\frac{n-3}2} |x|^{-1} dx,
	\end{equation}
	where $\kappa_d$ is the $(d-1)$-volume of $S^{d-1}$.
	To see this, write the Radon transform as $\mathcal R \delta (v,r) = \frac{\partial}{\partial r} \int_{\R^n} \delta(x) \chi_{(x, v) \leq r} dx$, then
\begin{align}
	\int_{S^{n-1}}\mathcal R \delta (v,r) d v
	&= \int_{\R^n} \delta(x) \frac{\partial}{\partial r} \int_{S^{n-1}} \chi_{(x, v) \leq r} d v d x,
\end{align}
which by rotational invariance can be computed as
\begin{align}
	\int_{S^{n-1}} \chi_{(x, v) \leq r} d v
		&= \int_{S^{n-1}} \chi_{ (e_1, v) \leq r/|x|} d v \\
		&= \kappa_{n-2} \int_{-1}^{r/|x|} (1-s^2)^{\frac{n-3}2} d s
\end{align}
if $r \leq |x|$, and $\kappa_{n-1}$ otherwise.
Take derivative with respect to $r$ to obtain \eqref{eq_integral_radon}.

Finally we get
	\[\int_{\R^n} |x|^{-n+2} f(|x|) \hat \delta(x) dx = \kappa_{n-2} \int_{\R^n} \frac {\delta(x)}{|x|} \frac 12 \int_{-\infty}^\infty r f(r) \mathcal F[ ( 1 - (r/|x|)^2 )_+^{\frac{n-3}2}] dr dx.\]
	For fixed $x \in \R^n \setminus \{0\}$, the function $\varphi_x(r) = ( 1 - (r/|x|)^2 )_+^{\frac{n-3}2}$ is bounded with compact support, even and non-increasing for $n\geq 3$, and the theorem follows from Lemma \ref{lem_decreasing_dim1} with $\psi(r) = r f(r)$ under the second condition.

	For the second part, notice that for $n \geq 9$, we have $\varphi_x \in C^3(\R)$ and we can use the second part of Lemma \ref{lem_decreasing_dim1}.
\end{proof}

\begin{proof}[Proof of Theorem \ref{thm_omega}]
	Since $f$ is absolutely continuous, it can be written as
	\[f(s) = \int_s^\infty \omega(\lambda) \lambda^{-n} d\lambda\]
for $s>0$, where $\omega(t) = - t^n f'(t)$.
	By hypothesis $\omega$ is non-negative, bounded and integrable.

The function
\[f(\|x\|_K) = \int_{\|x\|_K}^\infty \omega(\lambda) \lambda^{-n} d\lambda = \int_0^\infty \chi_{\lambda K}(x) \omega(\lambda) \lambda^{-n} d\lambda \]
	is locally integrable since $f(\|x\|_K) \leq (\max \omega) \|x\|_K^{1-n}$, and defines a tempered measure.

Take any non-negative Schwartz function $\varphi$.
	The previous bound on $f$ allows us to apply Tonelli's theorem with $|\hat\varphi(x)| \omega(\lambda) \lambda^{-n} \geq 0$, and conclude that $\hat\varphi(x) \omega(\lambda) \lambda^{-n}$ is integrable in $(x, \lambda) \in \R^n \times [0, \infty)$. By Fubini's theorem,
\begin{align}
	\int_{\R^n} f(\|x\|_K) \hat\varphi(x) d x
	&= \int_{\R^n} \int_0^\infty \chi_{\lambda K}(x) \hat\varphi(x) \omega(\lambda) \lambda^{-n} d\lambda d x \\
	&= \int_0^\infty \int_{\R^n} \chi_{\lambda K}(x) \hat\varphi(x) d x \omega(\lambda) \lambda^{-n} d\lambda \\
	&= \int_0^\infty \int_{\R^n} \widehat{\chi_{\lambda K}}(x) \varphi(x) d x \omega(\lambda) \lambda^{-n} d\lambda \\
	&= \int_0^\infty \int_{\R^n} [\chi_{K}(\lambda^{-1} \cdot)]^\wedge(x) \varphi(x) d x \omega(\lambda) \lambda^{-n} d\lambda \\
	&= \int_0^\infty \int_{\R^n} \widehat{\chi_{K}}(\lambda x) \varphi(x) d x \omega(\lambda) d\lambda.
\end{align}
	The integrand $\hat \chi_K(\lambda x) \varphi(x) \omega(\lambda)$ is clearly 
	integrable in $(x, \lambda) \in \R^n \times [0, \infty)$, so we can apply again Fubini's theorem again.
	Use the relation between Fourier and Radon transforms to get
\begin{align}
	\int_{\R^n} f(\|x\|_K) \hat\varphi(x) d x
	&= \frac 12 \int_{\R^n} \int_{-\infty}^\infty \widehat{A_{K, x/|x|}}(\lambda |x|) \omega(|\lambda|) d\lambda \varphi(x) d x\\
	&= \frac 12 \int_{\R^n}  \frac{\varphi(x)}{|x|} \int_{-\infty}^\infty  \widehat{A_{K, x/|x|}}(\lambda) \omega(|\lambda|/|x|) d\lambda d x.
\end{align}
	The function $A_{K, \xi}(\lambda)$ is even, bounded with compact support, and decreasing for $\lambda >0$.
	Also $\omega$ is integrable and $\omega(t) t^{-1} = -t^{n-1} f'(t)$ which by hypothesis is non-increasing, and we may apply Lemma \ref{lem_decreasing_dim1} to conclude that $\hat f$ is a non-negative distribution.

\end{proof}

\begin{proof}[Proof of Theorem \ref{thm_convex_curiosity}]
	First lets assume that $\varphi(x) = \chi_L, \psi(x) = \chi_K$ where $K$ is a convex body and $L$ is a radial body.
	By polar coordinates,
	\begin{align}
		\int_{\R^n} |x|^\alpha \varphi(x) \hat \psi(x) d x
		&= \int_L |x|^\alpha \hat \chi_K(x) d x \\
		&= \int_{S^{n-1}}\int_0^{\varrho_L(v)} r^{n-1+\alpha} \hat \chi_K(r v) d x \\
		&= \frac 12 \int_{S^{n-1}}\int_{-\infty}^\infty |t|^{n-1+\alpha} \chi_{[-\varrho_L(v),\varrho_L(v)]}(t) \hat A_{K,v}(t) d x,
	\end{align}
	where $\varrho_L(x) = \|x\|_K^{-1}$ is the radial function of $L$.
	Since $|t|^{n-1+\alpha} \chi_{[-\varrho_L(v),\varrho_L(v)]}(t)$ is integrable
	and $|t|^{-1} |t|^{n-1+\alpha} \chi_{[-\varrho_L(v),\varrho_L(v)]}(t) = |t|^{n-2+\alpha} \chi_{[-\varrho_L(v),\varrho_L(v)]}(t)$ is non-increasing for $t>0$,  we may apply Lemma \ref{lem_decreasing_dim1}.

	For the general case, we proceed as in Lemma \ref{lem_decreasing_dim1} using the layer-cake formula.
\end{proof}

\subsection*{Acknowledgments}
The author was supported by Grant RYC2021 - 031572 - I, funded by the Ministry of Science and Innovation / State Research Agency / 10.13039 / 501100011033 and by the E.U. Next Generation EU/Recovery, Transformation and Resilience Plan, and by Grant PID2022 - 136320NB - I00 funded by the Ministry of Science and Innovation.

The author wishes to thank A. Koldobsky for the very useful discussions that led to this article.

\bibliographystyle{abbrv}
\bibliography{../references}

\end{document}